\documentclass{amsart}
\usepackage[T1]{fontenc}
\usepackage[english]{babel}

\usepackage{enumitem}
	\setlist[enumerate]{label=\textit{(\roman*)},ref=(\roman*)}

\usepackage{amsmath}
\usepackage{amssymb,amsthm,amsfonts}
\usepackage{mathrsfs,mathtools,braket,esint}
\usepackage[mathscr]{euscript}
\usepackage{tikz}

	\numberwithin{equation}{section}
	\newtheorem{dfn}{Definition}[section]
	\newtheorem{thm}[dfn]{Theorem}

	\newtheorem{lemma}[dfn]{Lemma}
	\newtheorem{rmk}[dfn]{Remark}

	\newcommand{\call}[1]{\mathscr{#1}}
	\newcommand{\eps}{\varepsilon}
	\renewcommand{\epsilon}{\eps}
	\renewcommand{\phi}{\varphi}
	\renewcommand{\theta}{\vartheta}
	
	\def\N{\mathbb{N}}
	\def\R{\mathbb{R}}
	\def\Rd{\R^d}	
	\def\de{\mathrm{d}}
	
	\def\D{\mathrm{D}}
	\def\F{\call{F}}
	\def\Ld{\call{L}^d}	
	\def\sign{\mathrm{sign}}
	
	\newcommand{\vass}[1]{\left\lvert#1\right\rvert}
	\newcommand{\norm}[1]{\left\Vert#1\right\Vert}
	\renewcommand{\set}[1]{\left\{#1\right\}}
	\newcommand{\symdif}{\!\bigtriangleup\!}	
	\DeclareMathOperator{\Per}{Per}
	\newcommand{\PerK}{\Per_K}
\usepackage{amsrefs}
	
\usepackage{comment}
\usepackage{hyperref}

\title{Halfspaces minimise nonlocal perimeter: a~proof~\emph{via}~calibrations}
\author[V. Pagliari]{Valerio Pagliari}
\address{Dipartimento di Matematica,
		Universit\`a di Pisa, 
		Largo B. Pontecorvo 5, 56127 Pisa, Italy.
		E-mail: \href{mailto:pagliari@mail.dm.unipi.it}{\tt pagliari@mail.dm.unipi.it}}

\begin{document}

\begin{abstract}
We consider a nonlocal functional $J_K$
that may be regarded as a nonlocal version of the total variation.
More precisely, for any measurable function $u\colon \Rd \to \R$,
we define $J_K(u)$ as the integral of weighted differences of $u$.
The weight is encoded by a positive kernel $K$, possibly singular in the origin.
We study the minimisation of this energy under prescribed boundary conditions,
and we introduce a notion of calibration suited for this nonlocal problem.
Our first result shows that
the existence of a calibration is a sufficient condition for a function to be a minimiser.
As an application of this criterion, we prove that
halfspaces are the unique minimisers of $J_K$ in a ball, provided they are admissible competitors.
Finally, we outline how to exploit the optimality of hyperplanes
to recover a $\Gamma$-convergence result concerning the scaling limit of $J_K$.

\medskip
	{\footnotesize		
		
	}
\end{abstract}

\maketitle
	
	\section{Introduction}
	We consider the $d$-dimensional vector space $\Rd$ 
	equipped with the Euclidean inner product $\cdot$.
	In this note, we show that halfspaces are the unique local minimisers of the nonlocal functional
	\begin{align}\label{eq:JK}\begin{split}
	J_K(u;\Omega) & \coloneqq \frac{1}{2}\int_{\Omega}\int_{\Omega}K(y-x)\vass{u(y)-u(x)}\de y \de x \\
	& \quad +\int_{\Omega}\int_{\Omega^c}K(y-x)\vass{u(y)-u(x)}\de y \de x,
	\end{split}
	\end{align}
	where $\Omega \subset \Rd$ is a Lebesgue measurable set and $\Omega^c$ is its complement,
	while $u$ and $K$ are positive Lebesgue measurable functions on $\Rd$.
	Further hypotheses on the reference set $\Omega$
	and on the kernel $K$ are stated below, see Subsection \ref{sec:setup}.
	
	We recall that when $u=\chi_E$ is the characteristic function of the Lebesgue measurable set $E\subset\Rd$,
	that is $\chi_{E}(x)=1$ if $x\in E$ and $\chi_{E}(x)=0$ otherwise,
	then $J_K$ can be understood as a nonlocal perimeter of the set $E$ in $\Omega$.
	More generally, $J_K(u;\Omega)$ may be regarded as a nonlocal total variation of $u$ in $\Omega$.
	
	Nonlocal perimeters were firstly introduced by Caffarelli, Roquejoffre, and Savin \cite{CRS}
	to the purpose of describing phase field models
	that feature long-range space interactions.
	In their work, $K(x) = \left| x \right|^{-d-s}$, with $s\in(0,1)$.
	Subsequently, many authors have extended the analysis in several directions,
	and by now the literature has become vast;
	as a narrow list of papers that are more closely related to ours,
	we suggest that the interested reader may consult \cites{CV,ADM,CSV,MRT,CN} and the references therein.
	
	Let $B$ be the open unit ball in $\Rd$ with centre in the origin,
	put $\mathbb{S}^{d-1}\coloneqq \partial B$,
	and let $\call{L}^d$ be the $d$-dimensional Lebesgue measure.
	Our aim is proving the following:
	
	\begin{thm}\label{stm:piani}
		For all $\hat{n}\in\mathbb{S}^{d-1}$,
		we define $H\coloneqq \set{x\in\Rd : x\cdot \hat{n} > 0}$.
		Then,
		\[	J_K( \chi_H; B ) \leq J_K( v; B )	\]
		for all $\call{L}^d$-measurable $v\colon \Rd \to [0,1]$
		such that $v(x)=\chi_H(x)$ for $\call{L}^d$-a.e. $x\in B^c$.
		
		Moreover, for any other minimiser $u$ satisfying the same constraint,
		it holds $u(x)=\chi_H(x)$ $\call{L}^d$-a.e. $x\in \Rd$.
	\end{thm}
	
	The proof that we propose relies on a general criterion for minimality, see Theorem \ref{stm:plateau},
	which in turn involves a notion of calibration fitted to nonlocal problem at stake,
	see Definition \ref{stm:calib}.
	
	Let us outline the structure of this note.
	In the next Subsection, we precise the mathematical framework of this paper
	and we set the notations in use.
	Section \ref{sec:criterion} contains the definition of nonlocal calibration
	and the proof of Theorem \ref{stm:piani}.
	Lastly, in Section \ref{sec:Gammaconv}, as a possible application of our main result,
	we discuss its role in the analysis of the scaling limit of the functional $J_K$.
	
	\subsection{Set-up and notations}\label{sec:setup}
	We remind that we work in $\Rd$, the $d$-dimensional Euclidean space,
	endowed with the inner product $\cdot$
	and the associated norm $\left| \, \cdot \, \right|$.
	We let $\call{L}^d$  and $\call{H}^{d-1}$ be respectively
	the $d$-dimensional Lebesgue
	and the $(d-1)$-dimensional Hausdorff measure on $\Rd$.
	We shall henceforth omit to specify the measure w.r.t. which a set or a function is measurable,
	when the measure is $\Ld$ or the product $\Ld \otimes \Ld$ on $\Rd \times \Rd$;
	analogously, we shall use the expression ``a.e.'' in place of
	``$\Ld$-a.e.'' and of ``$\Ld \otimes \Ld$-a.e.''.
	If $u$ and $v$ are measurable functions, we shall also write ``$u=v$ in $E$''
	as a shorthand for ``$u(x)=v(x)$ for a.e. $x\in E$''.
	
	In this note, $\Omega\subset \Rd$ is an open and connected reference set
	such that $\call{L}^d(\Omega)\in(0,+\infty)$.
	Later on, in Section \ref{sec:Gammaconv}, some regularity on the boundary $\partial \Omega$ will be required.
	
	For what concerns the kernel $K\colon \Rd \to [0,+\infty]$,
	it is not restrictive to assume that is even, i.e.
	\[
	K(x) = K(-x) \qquad \text{a.e. } x\in\Rd.
	\]
	Besides, we suppose that
	\begin{equation} \label{eq:summK}
	\int_{\Rd} \left( 1 \wedge \left| x \right| \right) K(x) \de x < +\infty,
	\end{equation}
	where, if $t,s\in\R$, $t \wedge s$ equals the minimum between $t$ and $s$.
	This condition entails that $K \in L^1(B(0,r)^c)$ for all balls $B(0,r)$ with centre in the origin and radius $r>0$;
	in particular, $K$ might have a non-$L^1$ singularity in $0$.
	The main example of functions that fulfil \eqref{eq:summK} is given
	by fractional kernels \cites{CRS,L}, i.e. kernels of the form
	\[
	K(x) = \frac{a(x)}{ \vass{x}^{d+s} },
	\]
	where $a\colon \Rd \to \R$ is an even function such that $0<\lambda \leq a(x) \leq \Lambda$
	for some $\lambda,\Lambda\in \R$ and $s\in(0,1)$.
	
	A faster decay at infinity for $K$ will be needed in Section \ref{sec:Gammaconv},
	see \eqref{eq:fastK}.
	
	We are interested in a variational problem concerning $J_K$,
	to which we shall informally refer as \emph{Plateau's problem}.
	Precisely, given a Lebesgue measurable set $E_0 \subset \Rd$
	such that $J_K( \chi_{E_0}; \Omega)<+\infty$, we define the family
	\begin{equation}\label{eq:F}
	\call{F}\coloneqq\set{v\colon \Rd \to [0,1] : v \text{ is measurable and }
		v = \chi_{E_0} \text{ in } \Omega^c},
	\end{equation}
	and we address the minimisation of $J_K(\,\cdot\,;\Omega)$ in the class $\call{F}$;
	namely, we consider
	\begin{equation}\label{eq:plateau}
	\inf\set{J_K(v;\Omega) : v\in\call{F}}.
	\end{equation}
	
	\begin{rmk}[Truncation]
		For $s\in\R$, let us set $T(s) \coloneqq \big( (0 \vee  s) \wedge 1\big)$
		($t \vee s$ is the maximum between the real numbers $t$ and $s$).
		Observe that $T \circ \chi_{E_0} = \chi_{E_0}$ and $J_K( T \circ u; \Omega ) \leq J_K( u; \Omega )$,
		so the infimum in \eqref{eq:plateau} equals
		\[
		\inf\set{J_K(v;\Omega) :
			v\colon \Rd \to \R \text{ is measurable and }
			v = \chi_{E_0} \text{ in } \Omega^c
		}.
		\]
		We therefore see that choice of $\call{F}$ as the class of competitors is not restrictive.
	\end{rmk}
	
	\begin{rmk}[The class of competitors is nonempty]
		Standing our assumptions on $\Omega$,
		any set $E$ that has finite perimeter in $\Omega$ satisfies $J_K( \chi_E; \Omega ) <+\infty$,
		see \cites{MRT,BP}
		We shall recall the definition of finite perimeter set later in this Subsection.
	\end{rmk}
	
	As the functional $J_K(\,\cdot\,; \Omega)$ is convex,
	when $\Omega$ has finite measure,
	existence of solutions to \eqref{eq:plateau} can be established
	by the direct method of calculus of variations
	(see \cite{BP}; see also \cite{CSV} for an approach via approximation by smooth sets). 
	In particular, as consequence of the following coarea-type formula:
	\begin{equation}\label{eq:coarea}
	J_K(u;\Omega) = \int_{0}^{1}\PerK(\set{u>t};\Omega)\de t,
	\end{equation}	
	there always exists a minimiser which is a characteristic function.
	Indeed, for any $u\colon \Rd \to [0,1]$,	
	there exists $t^\ast\in\R$ such that $\PerK(\set{u>t^\ast}; \Omega)\leq J_K(u; \Omega)$,
	otherwise \eqref{eq:coarea} would be contradicted.
	Thus, if $u$ is a minimiser of \eqref{eq:plateau},
	then $\chi_{\set{u>t^\ast}}$ is minimising as well.
	
	Formula \eqref{eq:coarea} can be easily validated,
	see for instance \cites{CSV,CN}.
	The family of functionals on $L^1(\Omega)$ such that a generalised Coarea Formula holds
	was firstly introduced by Visintin \cite{V}. 
	
	It is well-known that
	existence of solutions to the classical counterpart of \eqref{eq:plateau}
	may be proved in the framework of geometric measure theory.
	We remind here some basic facts, while we refer to the monographs \cites{Ma,AFP}
	for a thorough treatment of the subject.
	
	We say that $u\colon \Omega \to \R$ is a \emph{function of bounded variation} in $\Omega$,
	and we write $u\in\mathrm{BV}(\Omega)$,
	if $u\in L^1(\Omega)$ and
	\[
	\vass{\D u}(\Omega) \coloneqq
	\sup\set{ \int_{\Omega} u(x)\mathrm{div}\zeta(x) \de x
		: \zeta\in C^{\infty}_c (\Rd;\Rd), \norm{\zeta}_{L^\infty}\leq 1}
	< + \infty.
	\]
	We dub $\vass{\D u}(\Omega)$ the \emph{total variation} of $u$ in $\Omega$.
	We also say that a measurable set $E$ is a \emph{set of finite perimeter} in $\Omega$
	when its characteristic function $\chi_E$ is a function of bounded variation in $\Omega$, and,
	in this case, we refer to
	$\Per(E;\Omega) \coloneqq \vass{\D \chi_E}(\Omega)$ as \emph{perimeter} of $E$ in $\Omega$.
	In this framework,
	the result that parallels the existence of solutions to \eqref{eq:plateau} reads as follows:
	there is a set $E$ with finite perimeter in $\Omega$ such that
	$\Per(E;\Omega)$ attains
	\begin{equation}\label{eq:cl-plateau}
	\inf\set{\vass{\D u}(\Omega) :
		u\colon \Rd \to [0,1] \text{ is measurable and }
		u = \chi_{E_0} \text{ in } \Omega^c}.
	\end{equation}
	
	Finite perimeter sets stand as measure-theoretic counterparts of smooth hypersurfaces.
	For example, we may equip them with an \emph{inner normal}:
	for any $x\in \mathrm{supp}\vass{\D\chi_E}$, we define
	\begin{equation}\label{eq:in-norm}
	\hat n(x)\coloneqq  \lim_{r\to 0^+}\frac{\D\chi_E(B(x,r))}{\vass{\D\chi_E}(B(x,r))},
	\end{equation}
	where $\D \chi_E$ is the distributional gradient of $\chi_E$
	and $B(x,r)$ is the open ball of centre $x$ and radius $r>0$.
	A fundamental result by De Giorgi \cite{D} states that
	\begin{equation*}\label{eq:per-Haus}
	\Per(E;\Omega)=\call{H}^{d-1}(\partial^\ast E\cap\Omega),
	\end{equation*}
	where
	\[
	\partial^\ast E \coloneqq \set{x\in \Rd : \hat n(x) \mbox{ exists and } \vass{\hat n(x)}=1}
	\]
	is the so-called \emph{reduced boundary} of $E$.
	In addition, for any $x\in\partial^\ast E$,
	\begin{equation}\label{eq:blowup}
	\frac{E-x}{r} \to \set{y\in\Rd : y\cdot \hat n(x) > 0} \quad\mbox{as } r\to 0^+ \mbox{ in } L^1_\mathrm{loc}(\Rd) .
	\end{equation}
	
	Once existence of solutions to \eqref{eq:cl-plateau} is on hand,
	a useful criterion to substantiate the minimality of a given competitor
	is provided by means of calibrations.
	The notion of calibration may be expressed in very general terms (see \cites{Mo,HL} and references therein);
	as far as we are concerned,
	we say that a (classical) \emph{calibration} for the finite perimeter set $E$
	is a divergence-free vector field  $\zeta\colon \Rd \to \Rd$
	such that $\left| \zeta(x) \right| \leq 1$ a.e. and
	$\zeta(x) = \hat{n}(x)$ for $\call{H}^{d-1}$-a.e. $x\in\partial^\ast E$.
	It can be shown that if the set $E$ admits a calibration,
	then its perimeter equals the infimum in \eqref{eq:cl-plateau}.
	The goal of the next Section is validating a nonlocal analogue of this principle.
	A similar analysis has been independently carried out by Cabr\'e in \cite{C}.
	In that work, the author proposes a notion of calibration akin to ours.
	He exploits it to establish optimality of nonlocal minimal graphs
	and to give as well a simplified proof of a result in \cite{CRS}
	stating that minimisers of the Plateau's problem satisfy
	a zero nonlocal mean curvature equation in the viscosity sense.
	
	\section{Minimality \emph{via} calibrations}\label{sec:criterion}
	In this Section,
	we propose a notion of calibration adapted to the current nonlocal setting,
	and we show that the existence of a calibration
	is a sufficient condition for a function $u$
	to minimise the energy $J_K$ w.r.t compact perturbations.
	Then, we show that halfspaces admit calibrations,
	and thus we infer their minimality.
	
	We remind that we assume that $\Rd\times\Rd$ is equipped with the product measure $\Ld \otimes \Ld$.
	
	\begin{dfn}\label{stm:calib}
		Let $u\colon \Rd \to [0,1]$ and	$\zeta\colon \Rd\times\Rd \to \R$ be measurable functions.
		We say that $\zeta$ is a \emph{nonlocal calibration} for $u$ if the following hold:
		\begin{enumerate}
			\item $\vass{\zeta(x,y)}\leq 1$ for a.e. $(x,y)\in\Rd\times\Rd$;
			\item for a.e. $x\in\Rd$,
			\begin{equation}\label{eq:div=0}
			\lim_{r\to 0^+}\int_{B(x,r)^c} K(y-x)\left(\zeta(y,x)-\zeta(x,y)\right)\de y = 0;
			\end{equation}
			\item for a.e. $(x,y)\in\Rd\times\Rd$ such that $u(x)\neq u(y)$,
			\begin{equation}\label{eq:nlnormal}
			\zeta(x,y)(u(y)-u(x))=\vass{u(y)-u(x)}.
			\end{equation}
		\end{enumerate}
	\end{dfn}
	
	The next remark collects some comments about the definition above.
	
	\begin{rmk}
		Let $\zeta\colon \Rd\times\Rd \to \R$ be a calibration  for $u\colon \Rd \to [0,1]$.
		\begin{enumerate}
			\item It is not restrictive to assume that $\zeta$ is antisymmetric:
			indeed, $\tilde\zeta(x,y)\coloneqq ( \zeta(x,y)-\zeta(y,x) )/2$ is a calibration for $u$ as well.
			\item In view of \eqref{eq:summK}, the integral in \eqref{eq:div=0} is convergent for each $r>0$.
			We can regard \eqref{eq:div=0} as a nonlocal counterpart
			of the vanishing divergence condition that is prescribed for classical calibrations.
			Such nonlocal gradient and divergence operators where introduced in \cite{GO},
			and they have already been exploited
			to study nonlocal perimeters by Maz\'on, Rossi, and Toledo in \cite{MRT},
			where the authors propose a notion of $K$-calibrable set
			in relation to a nonlocal Cheeger energy.
			\item Suppose that $u=\chi_E$ for some measurable $E\subset \Rd$.
			By \eqref{eq:nlnormal}, $\zeta$ must satisfy 
			\[
			\zeta(x,y) = \begin{cases}
			-1	& \text{if } x\in E, y\in E^c \\
			1	& \text{if } x\in E^c, y\in E.
			\end{cases}
			\]
			Heuristically, this means that 
			the calibration gives the sign of the inner product
			between the vector $y-x$ and the inner normal to $E$ at the ``crossing point'',
			provided the boundary of $E$ is sufficiently regular (see Figure \ref{fig:zeta}).
			Indeed, if we imagine to displace a particle from to $x$ and $y$,
			$\zeta$ equals $-1$ when the particle exits $E$,
			and it equals $1$ if the particles enters $E$.
		\end{enumerate}
	\end{rmk}
	
	\begin{figure}
		\caption{If $\zeta$ is a calibration for the set $E$
			(i.e. for $\chi_E$) and $x,y$ are as in the picture,
			then $\zeta(x,y)=-1$.}
		\label{fig:zeta}
		\centering
		\begin{tikzpicture}[scale=1.5,rotate=-20]
		\path (-0.1,-0.4) node{$E$};
		
		\draw (-1,0) .. controls (0,1) and (0,0) .. (1,0);
		\draw (1,-1) arc (90:270:-0.5);
		\draw (-0.5,-2) .. controls (0,-1) and (0,-1) .. (1,-1);
		\draw (-0.5,-2) .. controls (-1.5,-2) and (-0.5,-1) .. (-1,0);
		
		\path (0.4,-0.6) node{$x$};
		\filldraw [black] (0.5,-0.5) circle (0.5pt);
		%\filldraw [black] (1.5,0.5) circle (0.3pt);
		\path (1.6,0.6) node{$y$};
		\path (1.15,-0.26) node{$\hat n$};
		\draw[->] (0.5,-0.5) -- (1.5,0.5);
		\draw[->] (1,0) -- (1,-0.5);
		\draw[dashed] (0.3,0) -- (1.7,0);
		\end{tikzpicture}
	\end{figure}
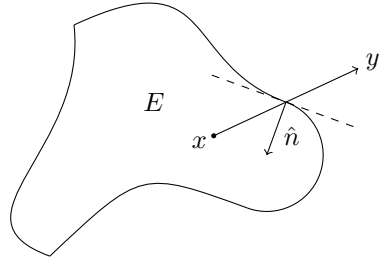
	
	Our criterion reads as follows:
	
	\begin{thm}\label{stm:plateau}
		Let $E_0\subset \Rd$ be a measurable set such that $J_K(\chi_{E_0};\Omega) <+\infty$,
		and let $\F$ be the family in \eqref{eq:F}.
		If for some $u\in\F$ there exists a calibration $\zeta$, then
		\[J_K(u;\Omega)\leq J_K(v;\Omega)\quad \text{for all } v\in\call{F}.\]
		
		Moreover, if $\tilde u\in\call{F}$ is another minimiser,
		then $\zeta$ is a calibration for $\tilde u$ as well.
	\end{thm}
	
	\begin{proof}
		By the definitions of $J_K(\,\cdot\,;\Omega)$, $\zeta$, and $\call{F}$,
		for any $v\in\call{F}$,
		\begin{equation}\label{eq:lowbound0}
		J_K(v;\Omega)\geq a(v) + b_1(v) + b_0,
		\end{equation}
		where
		\begin{equation*}\begin{gathered}
		a(v)\coloneqq \frac{1}{2}\int_{\Omega}\int_{\Omega}K(y-x)\zeta(x,y)(v(y)-v(x))\de y\de x,\\
		b_1(v)\coloneqq - \int_{\Omega}\int_{\Omega^c} K(y-x) \zeta(x,y) v(x)\de y \de x, \\
		b_0\coloneqq \int_{\Omega}\int_{\Omega^c} K(y-x) \zeta(x,y) \chi_{E_0}(y)\de y \de x.		
		\end{gathered}\end{equation*}
		Since it is not restrictive to assume that $J_K(v;\Omega)$ is finite,
		we can suppose that  $a(v)$, $b_1(v)$, and $b_0$  are finite as well.
		
		We claim that it suffices to prove that $a(v)=-b_1(v)$
		to grant the minimality of $u$.
		Indeed, $a(v)=-b_1(v)$ yields
		\begin{equation}\label{eq:lowbound}
		J_K(v;\Omega)\geq b_0 \quad\text{for all } v\in\call{F},
		\end{equation}
		and we remark that the lower bound $b_0$ is attained by $u$,
		because equality holds in \eqref{eq:lowbound0} for this function.
		Therefore, $u$ is a minimiser.
		
		Now, we prove that $a(v)=-b_1(v)$ for all $v\in\call{F}$.
		Recalling that we can assume $\zeta$ to be antisymmetric,
		we have
		\begin{equation*}
		a(v)=-\int_{\Omega}\int_{\Omega}K(y-x)\zeta(x,y)v(x)\de y\de x.
		\end{equation*}
		Also, \eqref{eq:div=0} gets
		\[\begin{split}
		0 & = -2\lim_{r\to 0^+}\int_{B(x,r)^c} K(y-x) \zeta(x,y)\de y \\
		& = -2\lim_{r\to 0^+}\int_{B(x,r)^c \cap \Omega} K(y-x) \zeta(x,y)\de y
		-2\int_{\Omega^c} K(y-x) \zeta(x,y)\de y,
		\end{split}\]
		whence
		\[
		a(v) = -\lim_{r\to 0^+} \int_\Omega\int_{B(x,r)^c \cap \Omega} K(y-x) \zeta(x,y) v(x)\de y \de x
		= -b_1(v).
		\]
		
		Next, let $\tilde u\in\call{F}$ be another minimiser of $J_K(\,\cdot\,;\Omega)$,
		that is  $J_K(\tilde u;\Omega)=b_0$.
		Our purpose is proving that for a.e. $(x,y)\in\Rd\times\Rd$ such that $ \tilde u(x)\neq\tilde u(y) $ it holds
		\begin{equation}\label{eq:optcond}
		\zeta(x,y)\left(\tilde u(y)-\tilde u(x)\right) = \vass{\tilde u(y)-\tilde u(x)}.
		\end{equation}		
		First of all, note the equality holds
		for a.e. $(x,y)\in\Omega^c \times \Omega^c$, because $u=\tilde u$ in $\Omega^c$. 
		Furthermore, from \eqref{eq:lowbound0} we have
		\[b_0=J_K(\tilde u;\Omega)\geq a(\tilde u)+b_1(\tilde u)+b_0 =b_0,\]
		thus
		\[\begin{split}
		\frac{1}{2}\int_{\Omega}\int_{\Omega}K(y-x)
		\left[\vass{\tilde u(y)-\tilde u(x)}-\zeta(x,y)(\tilde u(y)-\tilde u(x))\right]\de y \de x \\
		+\int_{\Omega}\int_{\Omega^c}K(y-x)
		\left[\vass{\tilde u(y)-\tilde u(x)}-\zeta(x,y)(\tilde u(y)-\tilde u(x))\right]\de y \de x = 0.
		\end{split}\]
		The integrand appearing in the previous identity is positive,
		therefore we deduce that \eqref{eq:optcond} is satisfied for a.e. $(x,y)\in\Omega\times\Rd$.
		Eventually, in the case $x\in\Omega^c$ and $y\in\Omega$,
		we achieve the conclusion by exploiting the antisymmetry of $\zeta$.		
	\end{proof}
	
	We take advantage of the previous theorem
	to prove that halfspaces are the unique local minimisers of $J_K(\,\cdot\,;B)$.
	This property has already been shown for fractional kernels in \cites{CRS,ADM}
	by means of a reflection argument, which in fact turns out to be effective
	whenever $K$ is radial and strictly decreasing \cite{BP}.
	Here, we are able to deal with the case when the kernel is neither monotone nor radial.
	
	We start with the following lemma,
	whose proof is a simple verification:
	\begin{lemma}
		Given $\hat{n}\in\mathbb{S}^{d-1}$, let us set
		\[	
		\zeta(x,y)\coloneqq \sign((y-x)\cdot \hat{n})
		\qquad\text{and}\qquad
		H\coloneqq \set{x\in\Rd : x\cdot \hat{n} > 0}.
		\]
		Then, $\zeta$ is a calibration for $\chi_H$.
	\end{lemma}
	
	Now, we prove Theorem \ref{stm:piani}.
	\begin{proof}[Proof of Theorem \ref{stm:piani}]
		In view of Theorem \ref{stm:calib} and of the Lemma above,
		we deduce that $\chi_H$ is a minimiser of the problem under consideration.
		Hence, we are left to prove uniqueness.
		
		Let $u\colon \Rd \to [0,1]$ be another minimiser. 
		The second assertion in Theorem \ref{stm:calib} grants that
		$\zeta(x,y)\coloneqq \sign((y-x)\cdot \hat{n})$ is a calibration for $u$ as well,
		so we get
		\begin{equation*}\label{eq:optcond-H}
		\sign((y-x)\cdot \hat{n})(u(y)-u(x))=\vass{u(y)-u(x)}\qquad \text{a.e. $(x,y)\in \Rd\times\Rd$.}
		\end{equation*}
		Let $N\subset \Rd\times\Rd$ be the negligible set
		of couples $(x,y)$ such that the previous equation does not hold;
		then, for all $(x,y) \in N^c \coloneqq \Rd \times \Rd \setminus N$
		such that $x\cdot\hat{n}<y\cdot\hat{n}$, we have $u(x) \leq u(y)$. 
		%\begin{equation}\label{eq:qo-mon}
		%\begin{cases}
		%u(x) = u(y) 		&\text{if } x\cdot\hat{n}=y\cdot\hat{n}\\
		%
		%\end{cases}.
		%\end{equation}
		
		We assert that, in fact, the implication that we have just obtained holds true everywhere in $\Rd\times\Rd$.
		To see this, let $\rho\in L^1(\Rd)$ be a positive, radial function such that $\int\rho =1$
		and whose support is contained in $B$.
		For $\eps \in (0,1)$, we consider the family
		$\rho_\epsilon(x)\coloneqq \eps^{-d}\rho\left( \eps^{-1} x \right)$
		and the convolutions $u_\epsilon\coloneqq\rho_\eps\ast u\colon \Rd \to [0,1]$. %which are uniformly continuous.
		Since $\int\rho =1$, for any couple $(x,y)\in \Rd\times\Rd$ we have
		\[\begin{split}
		u_\epsilon(y) - u_\epsilon(x) = &
		\int_{\Rd\times\Rd} \rho_\epsilon(\xi) \rho_\epsilon(\eta)
		\left[ u(y+\eta) - u(x+\xi) \right]
		\de \xi \de \eta \\
		= & \int_{B(0,\eps) \times B(0,\eps)} \rho_\epsilon(\xi) \rho_\epsilon(\eta)
		\left[ u(y+\eta) - u(x+\xi) \right] \de \xi \de \eta.
		\end{split}\]
		Let us suppose that $\delta \coloneqq (y-x)\cdot \hat{n} > 0$.
		If we choose $\epsilon\in(0,\delta/2)$, 
		we see that $ (x+\xi) \cdot \hat{n} < (y+\eta) \cdot \hat{n} $ for all $ \xi,\eta \in B(0,\eps) $,
		hence $u(x+\xi) \leq u(y+\eta)$ for a.e. $(\xi,\eta) \in B(0,\eps) \times B(0,\eps)$.
		Consequently, for all $(x,y)\in \Rd\times \Rd$, 
		it holds $u_\epsilon(x) \leq u_\epsilon(y)$ provided $\epsilon$ is small enough.
		Letting $\eps \to 0^+$, we find that
		\begin{equation}\label{eq:mon}
		u(x) \leq u(y)	\qquad\text{if } x\cdot\hat{n}<y\cdot\hat{n}
		\end{equation}
		
		Next, we focus on the superlevel sets of $u$: for $t\in(0,1)$, we define
		\[
		E_t \coloneqq \set{ x : u(x) > t},
		\]
		and we observe that if $(x,y)\in E_t  \times E_t^c$, it must be $x\cdot \hat n \geq y \cdot \hat n$,
		otherwise, by \eqref{eq:mon} we would have $u(x) \leq u(y)$.
		Therefore, there exists $\lambda_t \in \R$ such that
		$E_t \subset \set{x : x\cdot \hat{n} \geq \lambda_t}$ and
		$E_t^c \subset \set{y : y\cdot \hat{n} \leq \lambda_t}$,
		whence $\Ld( E_t\, \symdif\,  \set{x : x\cdot \hat{n} \geq \lambda_t} ) = 0$ for all $t\in(0,1)$.
		Recalling that it holds $u = \chi_H$ in $B^c$,
		we infer that $\lambda_t = 0$ and this gets
		\[
		\Ld( E_t \symdif  H ) = 0 \qquad \text{for all $t\in(0,1)$}.
		\]
		
		Summing up, we proved that $u\colon \Rd \to [0,1]$ is a function
		such that, for all $t\in(0,1)$,
		the superlevel set $E_t$ coincides with the halfspace $H$,
		up to a negligible set.
		To reach the conclusion,
		we let $\set{t_k}_{k\in \N}\subset (0,1)$ be a sequence that converges to $0$ when $k\to +\infty$.
		Because it holds
		\[
		\set{ x : u(x) = 0} = \bigcap_{k\in\N} E_{t_k}^c 
		\quad\text{and}\quad
		\set{ x : u(x) = 1} = \bigcap_{k\in\N} E_{1 - t_k},
		\]
		we see that
		$\Ld( \set{ x : u(x) = 0} \symdif  H^c ) = 0$ and $\Ld( \set{ x : u(x) = 1} \symdif  H) = 0$.
		Thus, $u = \chi_H$ in $\Rd$.
	\end{proof}
	
	\section{$\Gamma$-limit of the rescaled energy}\label{sec:Gammaconv}
	In this Section, we outline how to exploit Theorem \ref{stm:piani}
	to study the  limiting behaviour of certain rescalings of the energy $J_K$.
	In precise terms, we are interested in the $\Gamma$-convergence as $\eps \to 0^+$
	of $\set{ J_{K_\epsilon}(\,\cdot\,;\Omega) }$ 
	with respect to the $L^1_{\mathrm{loc}}(\Rd)$-convergence,
	where, for $\epsilon>0$, we let
	\[
	K_\epsilon(x) \coloneqq \frac{1}{\epsilon^d} K\left( \frac{x}{\eps} \right).
	\]
	In \cite{BP}, the analysis has already been carried out by Berendsen and the author of this note
	when $K$ is radial and strictly decreasing,
	but, as we concisely explain in the remainder of this note,
	the same arguments may be conveniently adapted to the current more general setting.
	We shall not deal with all the computations in depth,
	because our main interest here is how to take advantage
	of the minimality of halfspaces.
	This will be apparent in Lemma \ref{stm:sigmaK}.
	We refer to the works in the bibliography for the technical details.
	
	For the sake of completeness, we recall the following definition:
	
	\begin{dfn}[$\Gamma$-convergence]\label{stm:defGammac}
		Let $X$ be a set endowed with a notion of convergence
		and, for $\epsilon>0$, let $f_\epsilon\colon X\to[-\infty,+\infty]$ be a function.
		We say that the family $\set{f_\eps}$
		$\Gamma$-converges as $\epsilon\to0^+$ 	to the function $f_0\colon X\to [-\infty,+\infty]$
		w.r.t. the convergence in $X$ if
		\begin{enumerate}
			\item for any $x_0\in X$ and
			for any $\set{x_\eps}\subset X$ that converges to $x_0$,
			it holds
			\[f_0(x_0)\leq\liminf_{\epsilon\to 0}f_\eps(x_\eps);\] 
			\item for any $x_0\in X$ there exists $\set{x_\eps}\subset X$
			that converges to $x_0$ with the property that
			\[\limsup_{\epsilon\to 0}f_\eps(x_\eps)\leq f_0(x_0).\] 
		\end{enumerate}
	\end{dfn}
	
	When $u\colon \Rd \to  [0,1]$ is a measurable function, 
	let us define
	\begin{gather*}
	J_\eps^1(u; \Omega) \coloneqq\frac{1}{2}\int_\Omega\int_\Omega K_\eps(y-x)\vass{u(y)-u(x)}\de y\de x, \\
	J_\eps^2(u; \Omega) \coloneqq\int_{\Omega}\int_{\Omega^c}K_\eps(y-x)\vass{u(y)-u(x)}\de y \de x,\\
	J_\eps(u; \Omega) \coloneqq J_\eps^1(u; \Omega) + J_\eps^2(u; \Omega).
	\end{gather*}
	Observe that, according to the notation in \eqref{eq:JK}, $J_\eps = J_{K_\eps}$.
	We also introduce the limit functional
	\begin{equation*}
	J_0(u;\Omega) \coloneqq 
	\begin{cases}
	\displaystyle{
		\frac{1}{2}\int_{\Rd} K(z) \left( \int_\Omega \vass{ z \cdot \D u} \right) \de z
	}
	& \text{if } u\in\mathrm{BV}(\Omega), \\
	+\infty & \text{otherwise}.
	\end{cases}
	\end{equation*}
	Our goal is proving the following:
	
	\begin{thm}[$\Gamma$-convergence of the rescaled energy]\label{stm:Gconv}
		Let $\Omega \subset \Rd$ be an open, connected, and bounded set
		with Lipschitz boundary.
		Let also $K\colon \Rd \to [0,+\infty)$ be an even function such that
		\begin{equation}\label{eq:fastK}
		\int_{\Rd} K(x)\vass{x} dx < +\infty.
		\end{equation}
		Then, for any measurable $u\colon \Rd \to [0,1]$ the following hold:
		\begin{enumerate}
			\item\label{stm:lower}
			For any family $\set{u_\eps}$
			that converges to $u$ in $L^1_\mathrm{loc}(\Rd)$,
			we have
			\[
			J_0(u;\Omega) \leq \liminf_{\epsilon\to 0^+} \frac{1}{\eps} J^1_\epsilon(u_\epsilon;\Omega).
			\]
			\item\label{stm:upper}
			There exists a family $\set{u_\eps}$ that converges to $u$ in $L^1_\mathrm{loc}(\Rd)$
			such that
			\[
			\limsup_{\epsilon\to 0^+} \frac{1}{\eps} J_\epsilon(u_\epsilon;\Omega) \leq J_0(u;\Omega).
			\]
		\end{enumerate}
	\end{thm}
	
	We remark that, being $J_\eps^2(\,\cdot\,;\Omega)$ positive,
	Theorem \ref{stm:Gconv} entails
	the $\Gamma$-convergence of $\set{J_\epsilon(\,\cdot\,;\Omega)}$ to $J_0(\,\cdot\,;\Omega)$
	w.r.t. the $L^1_\mathrm{loc}(\Rd)$-convergence.
	Also, note that \eqref{eq:fastK} prescribes a condition that is more stringent than \eqref{eq:summK}.
	
	Several results about the asymptotics of functionals akin to $J_\epsilon$ have been considered in the literature
	\cites{BBM,D,P,MRT,AB}; in particular,
	we wish to mention the following one by Ponce:
	
	\begin{thm}[Corollary 2 and Theorem 8 in \cite{P}]\label{stm:ponce}
		Let $\Omega\subset \Rd$ be an open bounded set with Lipschitz boundary
		and let $u\in \mathrm{BV}(\Omega)$.
		If \eqref{eq:fastK} holds, then
		\begin{equation}\label{eq:pointlim}
		\lim_{\epsilon\to 0^+} \frac{1}{\eps} J^1_\eps(u;\Omega) 
		= J_0(u;\Omega)
		\end{equation}
		Moreover, $J_0(\,\cdot\,;\Omega)$ is the $\Gamma$-limit as $\epsilon\to 0^+$
		of $\set{\epsilon^{-1} J_\epsilon(\,\cdot\,;\Omega)}$ w.r.t. the $L^1(\Omega)$-topology.
	\end{thm}
	
	We discuss separately the proofs of statements \ref{stm:lower} and \ref{stm:upper} in Theorem \ref{stm:Gconv}.
	Preliminarly, we remark that
	we only need to study the $\Gamma$-convergence of $J_\epsilon$
	regarded as a functional on measurable sets,
	namely, for $E\subset \Rd$ measurable, we consider
	\begin{gather*}
	J^i_\epsilon(E;\Omega) \coloneqq J^i_\epsilon(\chi_E;\Omega)
	\qquad\text{for } i=1,2, \\
	J_\eps(E;\Omega) \coloneqq J_\epsilon(\chi_E;\Omega),
	\end{gather*}
	and the limit functional
	\[
	J_0(E;\Omega) \coloneqq J_0(\chi_E;\Omega).
	\]
	Indeed, by appealing to results by Chambolle, Giacomini, and Lussardi \cite[Propositions 3.4 and 3.5]{CGL},
	it is possible to recover  the $\Gamma$-convergence of $J_\epsilon$
	as a functional on measurable functions from the analysis of the restrictions;
	this is mainly due to convexity and to the validity of Coarea Formulas.
	
	So, as for the $\Gamma$-upper limit inequality, we need to show that,
	for any given measurable $E\subset \Rd$, there exists a family $\set{E_\eps}$
	that converges to $E$ in $L^1_\mathrm{loc}(\Rd)$ as $\epsilon\to0^+$ such that
	\begin{equation*}
	\limsup_{\epsilon\to 0^+} \frac{1}{\eps} J_\epsilon^1(E_\epsilon;\Omega) \leq 	J_0(E;\Omega).
	\end{equation*}
	Hereafter, by saying that
	the family of sets $\set{E_\eps}$ converges to $E$ in $L^1_\mathrm{loc}(\Rd)$,
	we mean that $\chi_{E_\eps} \to \chi_E$ in $L^1_\mathrm{loc}(\Rd)$.
	
	The desired inequality may be achieved as in \cite{BP}
	by reasoning on a class of sets $\call{D}$
	which is dense w.r.t. the energy $J_0$ among all measurable sets.
	We omit the details, since Theorem \ref{stm:piani} plays no role in this step. 
	
	Now we turn to the proof of the $\Gamma$-lower limit inequality.
	Our task is proving that,
	for any given measurable $E\subset \Rd$ and for any family $\set{E_\eps}$
	that converges to $E$ in $L^1_\mathrm{loc}(\Rd)$ as $\epsilon\to0^+$, it holds
	\begin{equation}\label{eq:lowerlim}
	J_0(E;\Omega) \leq \liminf_{\epsilon\to 0^+} \frac{1}{\eps} J_\epsilon^1(E_\epsilon;\Omega).
	\end{equation}
	
	In \cite{P}, the approach to the $\Gamma$-lower limit inequality relies on
	representation formulas for the relaxations of a certain class of integral functionals.
	Here, following \cite{BP}, we propose a strategy
	which combines the pointwise limit \eqref{eq:pointlim} and Theorem \ref{stm:piani}.
	
	Observe that we can write
	\begin{equation*}
	J_0(E;\Omega) \coloneqq 
	\begin{cases}
	\displaystyle{
		\int_{\partial^\ast E \cap \Omega} \sigma_K(\hat{n}(x)) \de \call{H}^{d-1}(x)
	}
	& \text{if $E$ is a finite perimeter set in $\Omega$,} \\
	+\infty 	& \text{otherwise},
	\end{cases}
	\end{equation*}
	where $\hat{n}\colon \partial^\ast E \to \mathbb{S}^{d-1}$ is
	the measure-theoretic inner normal of $E$ (recall \eqref{eq:in-norm})
	and $\sigma_K\colon \Rd \to [0,+\infty)$ is the anisotropic norm
	\begin{equation}
	\sigma_K(p) \coloneqq \frac{1}{2}\int_{\Rd} K(z)\vass{z\cdot p}\de z,
	\qquad\text{for } p\in\Rd.
	\end{equation}
	
	\begin{rmk}[The radial case \cite{BP}]
		When $K$ is radial, $J_0$ coincides with De Giorgi's perimeter,
		up to a multiplicative constant that depends on $K$ and on $d$.
		Indeed, if $K(x) = \bar K (\vass{x})$ for some $\bar K\colon [0,+\infty) \to [0,+\infty)$,
		for any $\hat p\in\mathbb{S}^{d-1}$, we have that
		\begin{align*}
		\sigma_K(\hat p) & = \frac{1}{2} \left(\int_{0}^{+\infty} \bar K(r) r^d \de r \right) 
		\int_{\mathbb{S}^{d-1}} \vass{e\cdot \hat p} \de \call{H}^{d-1}(e) \\
		& = \frac{1}{2} \left(\int_{\Rd} K(x)\vass{x} \de x\right)
		\fint_{\mathbb{S}^{d-1}} \vass{e\cdot e_d} \de \call{H}^{d-1}(e),
		\end{align*}
		where $e_d \coloneqq (0,\dots,0,1)$ is the last element of the canonical basis.
	\end{rmk}
	
	By a blow-up argument \textit{\`a la} Fonseca-M\"uller \cite{FM}
	that has already been applied to similar problems \cites{AB,ADM},
	it turns out that the $\Gamma$-lower limit inequality \eqref{eq:lowerlim} holds
	as soon as one characterises the norm $\sigma_K$ 
	in terms of the evaluation on halfspaces
	of the $\Gamma$-inferior limit of $\epsilon^{-1} J_\epsilon(\,\cdot\,;B)$.
	Precisely, we need to validate the following:
	
	\begin{lemma}\label{stm:sigmaK}
		For any $\hat{p}\in\mathbb{S}^{d-1}$,
		\[
		\sigma_K (\hat{p})
		= \inf\set{
			\liminf_{\epsilon\to 0^+} \frac{1}{\omega_{d-1} \eps} J^1_\eps(E_\epsilon;B) :
			E_\eps \to H_{\hat{p}} \text{ in } L^1(B)
		},
		\]
		where $\omega_{d-1}$ is the $(d-1)$-dimensional Lebesgue measure
		of the unit ball in $\R^{d-1}$,
		and $H_{\hat{p}} \coloneqq \set{ x\in\Rd : x\cdot \hat{p} > 0 }$.
	\end{lemma}
	
	It is in the proof of this Lemma that Theorem \ref{stm:piani} comes into play.
	
	\begin{proof}[Proof of Lemma \ref{stm:sigmaK}]
		For $\hat{p}\in  \mathbb{S}^{d-1}$, let us set
		\begin{equation}\label{eq:sigma'K}
		\sigma'_K(\hat{p}) \coloneqq \inf\set{
			\liminf_{\epsilon\to 0^+} \frac{1}{\omega_{d-1} \eps} J^1_\eps(E_\epsilon;B) :
			E_\eps \to H_{\hat{p}} \text{ in } L^1(B)
		}.
		\end{equation}
		By \eqref{eq:pointlim}, we now that
		\begin{equation}\label{eq:sigmaK-P}
		\sigma_K(\hat{p}) = \lim_{\epsilon\to 0^+} \frac{1}{\omega_{d-1}\eps} J^1_\eps(H_{\hat{p}};B),
		\end{equation}
		hence $\sigma_K (\hat{p}) \geq \sigma'_K(\hat{p})$.
		
		To the purpose of proving the reverse inequality,
		we introduce a third function $\sigma''_K$ and we show that
		$\sigma_K \leq \sigma''_K  \leq \sigma'_K$.
		So, for $\hat{p}\in \mathbb{S}^{d-1}$ and $\delta\in(0,1)$, we let
		\[
		\sigma''_K(\hat{p}) \coloneqq \inf\set{
			\liminf_{\epsilon\to 0^+} \frac{1}{ \omega_{d-1} \eps} J^1_\eps(E_\epsilon;B) :
			E_\eps \to H_{\hat{p}} \text{ in } L^1(B)
			\text{ and }
			E_\eps \symdif H_{\hat{p}} \subset B_{1-\delta}
		},
		\]
		where $B_{1-\delta} \coloneqq B(0,1-\delta)$
		and $E_\eps \symdif H_{\hat{p}}$ is the symmetric difference between $E_\eps$ and $H_{\hat{p}}$.
		We decide not use a notation that
		exhibits the dependence of $\sigma''_K$ on the parameter $\delta$
		because \textit{a posteriori} the values of $\sigma''_K$ are not influenced by it.
		
		We firstly show that $\sigma_K \leq \sigma''_K$.
		Let ${E_\eps}$ be a family of measurable subsets of $\Rd$ such that
		$E_\eps \cap B^c = H_{\hat{p}}\cap B^c$
		and that $E_\eps \to H_{\hat{p}}$ in $L^1(B)$.
		By Theorem \ref{stm:piani},
		we have that
		\begin{align*}
		0 & \leq J_\epsilon(E_\eps; B) - J_\epsilon(H_{\hat{p}}; B)  \\
		& =  J^1_\epsilon(E_\eps; B) - J^1_\epsilon(H_{\hat{p}}; B)
		- \left[ J^2_\epsilon(E_\eps; B) - J^2_\epsilon(H_{\hat{p}}; B) \right].
		\end{align*}
		If we also assume that $E_\eps \symdif H_{\hat{p}} \subset B_{1-\delta}$,
		we see that
		\begin{multline*}
		J^2_\epsilon(E_\eps; B) - J^2_\epsilon(H_{\hat{p}}; B)
		\\ = \int_{E_\eps \cap B_{1-\delta}}\int_{H_{\hat{p}} \cap B^c} K_\eps(y-x) \de y \de x
		-
		\int_{H_{\hat{p}} \cap B_{1-\delta}}\int_{H_{\hat{p}} \cap B^c} K_\eps(y-x) \de y \de x
		\end{multline*}
		and hence, noticing that $\vass{y-x}\geq\delta$ if $x\in B^c$ and $y\in B_{1-\delta}$, 
		\begin{align*}
		\frac{1}{\eps}\vass{ J^2_\epsilon(E_\eps; B) - J^2_\epsilon(H_{\hat{p}}; B) }
		& \leq \frac{2}{\delta}
		\int_{ E_\eps \symdif H _{\hat{p}}} \int_{ B^c } 
		K_\eps (y-x) \frac{\vass{y-x}}{\eps}\de y \de x \\
		& \leq  \frac{2}{\delta} \Ld (E_\eps \symdif H _{\hat{p}})
		\int_{ \Rd } K(z) \vass{z} \de z.
		\end{align*}
		By our choice of $\set{E_\eps}$ and \eqref{eq:fastK}, this yields
		\[ 
		\lim_{\epsilon\to 0^+} \frac{1}{\eps} \vass{ J^2_\epsilon(E_\eps; B) - J^2_\epsilon(H_{\hat{p}}; B) } = 0,
		\]
		whence 
		\begin{align*}
		0 & \leq \liminf_{\eps \to 0^+}   \frac{1}{\eps} \left[ J_\epsilon(E_\eps; B) - J_\epsilon(H_{\hat{p}}; B) \right]  \\
		& =  \liminf_{\eps \to 0^+}   \frac{1}{\eps} \left[ J^1_\epsilon(E_\eps; B) - J^1_\epsilon(H_{\hat{p}}; B) \right].
		\end{align*}
		Recalling \eqref{eq:sigmaK-P} and the definition of $\sigma''_K$,
		we deduce $\sigma_K(\hat{p}) \leq \sigma''_K(\hat{p})$.
		
		To conclude, we are left to show that $\sigma''_K \leq \sigma'_K$.
		This may be done as in the proof of \cite[Lemma 3.11]{BP}
		by means of a suitable ``gluing'' lemma (see also \cite{ADM}).
	\end{proof}

\subsection*{Acknowledgements}
	The author warmly thanks Matteo Novaga for suggesting the problem
	and for providing insights about it.
	Part of this paper was written during a stay of the author
	at the Centre de Math\'ematiques Appliqu\'ees (CMAP) of the \'Ecole Polytechnique in Paris.
	The author thanks the institution for hospitality and
	the Unione Matematica Italiana (UMI) for partially funding the stay
	\emph{via} the \emph{UMI Grants for Ph.D. students}.

\end{document}